\documentclass [11pt]{article}

\usepackage{amsmath,amsfonts,amsthm,amssymb,fullpage}
\usepackage{graphics,color,url}
\usepackage{graphicx}
\usepackage{epsfig}

\newtheorem{definition}{Definition}[section]

\newtheorem{lemma}{Lemma}[section]

\newtheorem{theorem}{Theorem}[section]
\newtheorem{conjecture}{Conjecture}[section]

\title{On a Conjecture of Randi\'{c} Index and Graph Radius}
\author{\\Hanyuan Deng\footnote{College of Mathematics and Computer Science, Key
Laboratory of High Performance Computing and Stochastic Information
Processing (Ministry of Education of China), Hunan Normal
University, Changsha, Hunan 410081, China. Email: {\tt hydeng@hunnu.edu.cn}.} \and \\Zikai Tang\footnotemark[1] \and \\Jie Zhang\footnote{Department of Computer Science, Aarhus University, Denmark. Email: {\tt csjiezhang@gmail.com}.} \thanks{The author acknowledges support from the Danish National Research Foundation and the National Science Foundation of China (grant 61061130540) for the Sino-Danish Center for the Theory of Interactive Computation, and from the Center for research in Foundations of Electronic Markets (CFEM), supported
by the Danish Council for Strategic Research.}}
\date{}

\begin{document}

\maketitle

\begin{abstract}
The {\em Randi\'{c} index} $R(G)$ of a graph  $G$ is defined
as the sum of $(d_i d_j)^{-\frac{1}{2}}$ over all edges $v_i v_j$ of $G$, where $d_i$ is the degree of the vertex $v_i$ in $G$. The {\em radius} $r(G)$ of a graph $G$ is the minimum graph eccentricity of any graph vertex in $G$. \cite{fa} conjectures $R(G) \ge r(G)-1$ for all connected graph $G$. A stronger version, $R(G) \ge r(G)$, is conjectured by \cite{ch} for all connected graphs except even paths. 
In this paper, we make use of {\em Harmonic index} $H(G)$, which is defined as the sum of $\frac{2}{d_i+d_j}$ over all edges $v_i v_j$ of $G$, to show that $R(G) \ge r(G)-\frac{31}{105}(k-1)$ for any graph with cyclomatic number $k\ge 1$, and $R(T)> r(T)+\frac{1}{15}$ for any tree except even paths. These results improve and strengthen the known results on these conjectures.
\end{abstract}

\setlength{\baselineskip}{.45cm}

\section{Introduction}

Topological indices are numerical parameters of a graph which characterize the topological structure of the graph and are usually graph invariants. The {\em Randi\'{c} index}, one of the most well-known topological indices, is introduced by~\cite{Randic} and is generalized by~\cite{B.E}. It studies the branching property of graphs. Since its appearance, tremendously attention has been focused on the upper and lower bounds of the index. ~\cite{B.E} prove that the  Randi\'{c} index of a graph of order $n$ without isolated vertices is at least $\sqrt{n-1}$; they leave the open problem that the minimum value of the  Randi\'{c} index for graphs $G$ with given minimum degree $\delta(G)$.~\cite{DFR} answer this question for $\delta(G)=2$, thus partially solve the problem. Furthermore, they prove a best possible lower bound on the Randi\'{c} index of a triangle-free graph $G$ with given minimum degree $\delta(G)$.~\cite{BBG} build up a technique to determine the maximal  Randi\'{c}  index of a tree with a specified number of vertices and leaves. Reviews of mathematical properties of the Randi\'{c}  index refer to~\cite{GutmanFurtula}, \cite{LiGutman}.

On the other side, \cite{fa} and \cite{ch} conjecture that the  Randi\'{c}  index can be lower bounded in terms of the graph radius. In this paper, we improve and strengthen the known results on these conjectures by studying the relationship between the Harmonic index and graph radius.

 The Harmonic index is defined by \cite{b5}. \cite{b6} consider the relationship between the Harmonic index and graph eigenvalues. \cite{b12} find the minimum and maximum values
of the Harmonic index for simple connected graphs and trees, and
characterize the corresponding extremal graphs. \cite{d} consider the relationship between the Harmonic index $H(G)$ and the chromatic number $\chi(G)$ and prove that $\chi(G)\leq 2H(G)$. It strengthens a conjecture of the Randi\'{c} index and the chromatic number which is based on the system AutoGraphiX and is proved by \cite{b7}.  \cite{dtw} give a best possible lower bound for the Harmonic index of a graph and a triangle-free graph with minimum degree no less than two and characterize the extremal graphs, respectively. 

{\em Organization.} In section 2, we introduce neccessary notations used in this paper, and state our main results. In section 3, 4 and 5, we prove the main results. We conclude our work in Section 6.

\section{Preliminary}

Let $G$ be a simple undirected graph with vertex set
$V(G)=\{v_1,\dots,v_n\}$ and edge set $E(G)=\{e_1,\dots,e_m\}$.
Let's denote edge $v_i v_j \in E$ if $v_i$ and $v_j$ are adjacent in
graph $G$. Let $d_i$ be the degree of vertex $v_i$, $i=1,\dots,n$.
Unless otherwise specified we focus on non-empty connected graph
throughout the paper.  A pendant vertex is a vertex of degree one. A
path with even (odd) vertices is called an even (odd) path. A cycle
with even (odd) vertices is called an even (odd) cycle. The neighborhood $N(v_i)$ is the set of vertices adjacent to $v_i$.  The distance $\rho(v_i, v_j)$ is the number of edges in a shortest path
connecting $v_i$ and $v_j$ in $G$. The {\em radius} of a graph $G$ is the minimum eccentricity of any vertex; that is, $r(G)=\min\limits_{v_i\in V}\max\limits_{v_j\in V}\rho(v_i,v_j)$. The {\em cyclomatic number} $k$ of a graph $G$, also known as the {\em circuit rank}, is the minimum number of edges to remove from the graph to make it cycle-free; that is, $k=|E|-|V|+1$. Obviously, the cyclomatic number of unicyclic, bicyclic and tricyclic graphs are 1, 2 and 3, resectively. 

The celebrated {\em  Randi\'{c} index} of graph $G$ is
introduced by~\cite{Randic}.

\begin{definition}
Given any graph $G$, the Randi\'{c} index of $G$ is
\begin{align*}
R(G)=\sum_{v_i v_j \in E} \frac{1}{\sqrt{d_id_j}},
\end{align*}
where the sum is over all edges $v_iv_j$ of the graph $G$.
\end{definition}

The following interesting conjecture is proposed by \cite{fa}.

\begin{conjecture}[\cite{fa}]
For any connected graph $G$,
$R(G)\geq r(G)-1$.
\end{conjecture}

\cite{ch} prove that $R(T)\geq r(T)+\sqrt{2}-\frac{3}{2}$ for any tree $T$, and $R(T)\geq r(T)$ for any tree $T$ except even paths. \cite{lg}, and \cite{yl} prove that the conjecture is true for unicyclic, bicyclic and tricyclic graphs. \cite{ch} also propose the following stronger version of the conjecture.

\begin{conjecture}[\cite{ch}] 
For any connected graph $G$ except
even paths, $R(G)\geq r(G)$.
\end{conjecture}

We confirm that the conjectures are true for some graphs by studying the relationship between Harmonic index and graph radius. The {\em Harmonic index} is defined by  \cite{b5} as follows.

\begin{definition}
Given any graph $G$, the Harmonic index of $G$ is
\begin{align*}
H(G)=\sum_{v_i v_j \in E} \frac{2}{d_i+d_j},
\end{align*}
where the sum is over all edges $v_iv_j$ of the graph $G$.
\end{definition}

For any path $P_n$ with $n\geq 3$ vertices, it is easy to check that
$H(P_n)=\frac{n}{2}-\frac{1}{6}$ and $r(P_n)=\lfloor\frac{n}{2}\rfloor$. Therefore,
\begin{align}\label{HarmonicPath}
H(P_n)=
\left\{
\begin{array}{ll}
r(P_n)-\frac{1}{6},  \,\,\,\,\,\ & \hbox{if $n$ is even;} \\
r(P_n)+\frac{1}{3}, \,\,\,\,\,\ & \hbox{if $n$ is odd.}
\end{array}
\right.
\end{align}

Since $\sqrt{xy} \le \frac{x+y}{2}, \forall x,y\in \mathbb{R}^{+}$, we
obtain $R(G)\ge H(G)$ for any graph $G$. 

Our main results have the following three aspects, which improve and strengthen the known results on Conjectures 1 and 2.

\begin{enumerate}
\item 
For all trees $T$ except even paths, $H(T) > r(T)+\frac{1}{15} > r(T)$. 

We thus partially confirm Conjecture 2 and improve the result of \cite{ch} for trees.

\item
For all unicyclic graphs $G$, $H(G)\geq r(G)$. The equality holds if and
only if $G$ is an even cycle.

We thus confirm the Conjecture 2 for unicyclic graphs.

\item
For all graphs $G$ with cyclomatic number $k\geq 1$,
$H(G) \ge r(G) - \frac{31}{105}(k-1)$. In particular, $H(G)>r(G)-1$ for all graphs with cyclomatic number no more than 4.

We thus confirm that Conjecture 1 is true not only for trees, unicyclic, bicyclic and tricyclic graphs, but also for graphs have cyclomatic number 4. In addition, this result implies the inequality in Conjecture 1 strictly holds for all graphs with cyclomatic number no more than 4. 
\end{enumerate}

\section{The Harmonic index and the radius of a tree}

We first show that adding a pendant edge to a graph $G$ strictly
increases its Harmonic index.

\begin{lemma}
If $G_0$ is obtained by adding a pendant edge $ v_i v_{n+1}$ to a graph $G$,
where $v_i \in V(G)$, then $H(G_0)>H(G)$.
\end{lemma}

\begin{proof}
Note that $\frac{1}{x+1}-\frac{1}{x}$ is
increasing in $x>1$. According to the definition of the Harmonic index, we
have
\begin{align*}
H(G_0) - H(G)=&\frac{2}{d_i+1+d_{n+1}}+\sum\limits_{v_j\in
N(v_i)\setminus\{v_{n+1}\}} \left( \frac{2}{d_i+1+d_j}-\frac{2}{d_i+d_j} \right)\\
\geq& \frac{2}{d_i+2}+d_i \left( \frac{2}{d_i+2}-\frac{2}{d_i+1} \right)\\
=& \frac{2d_i}{(d_i+1)(d_i+2)}>0.
\end{align*}
\end{proof}

\begin{theorem}\label{tree}
For all trees $T$ except even paths, $H(T) > r(T)+\frac{1}{15}$.
\end{theorem}

\begin{proof} 
Without loss of generality, assume the diameter of the tree $T$ is $k-1$, and $P_k$ is the longest shortest path of $T$. So $r(T)=r(P_k)$. There are two cases to consider.

\begin{enumerate}
\item
if $k$ is odd, according to~\eqref{HarmonicPath}, $H(P_k)=r(P_k)+\frac{1}{3}$. In addition, the tree $T$ can be derived from path $P_k$ by adding pendent edges step by step. According to Lemma 1, 
\begin{align*}
H(T) > H(P_k) = r(P_k)+\frac{1}{3} = r(T)+\frac{1}{3} > r(T)+\frac{1}{15}.
\end{align*}

\item
if $k$ is even, according to~\eqref{HarmonicPath}, $H(P_k)=r(P_k)-\frac{1}{6}$. Let the tree $T_{0}$ be a subgraph of $T$, and is obtained by adding one pendent edge to $P_k$ but retaining its diameter; that is, the newly added pendent edge is not incident to the pendent vertices of $P_k$.
\begin{itemize}
\item
If the newly added pendent edge is adjacent to the pendent edges of $P_k$, then by simple calculation we get 
\begin{align*}
H(T_{0}) = H(P_k) + \frac{7}{30}.
\end{align*}

\item
If the newly added pendent edge is not adjacent to the pendent edges of $P_k$, then by simple calculation we get 
\begin{align*}
H(T_{0}) =H(P_k) + \frac{3}{10}.
\end{align*}
\end{itemize}

In all, $H(T_{0})\ge H(P_k)+\frac{7}{30}$. By the same argument in case 1, we derive the tree $T$ from $T_{0}$ by adding pendent edges step by step and get 
\begin{align*}
H(T) > H(T_{0}) \ge H(P_k)+\frac{7}{30} = r(P_k) - \frac{1}{6} + \frac{7}{30} = r(T) + \frac{1}{15}.
\end{align*}
\end{enumerate}
\end{proof}

\section{The Harmonic index and the radius of a unicyclic graph}

In this section, we discuss the Harmonic index and the radius of
unicyclic graphs.

\begin{theorem}\label{unicyclic}
For all unicyclic graphs $G$, $H(G)\geq r(G)$. The equality holds if and
only if $G$ is an even cycle.
\end{theorem}

\begin{proof}
Let $C=u_1 u_2 \cdots u_l u_1$ be the unique cycle of
$G$, where $l\ge 3$, and $|V(G)|=n$. If $G=C$ is a cycle, then $H(G)=\frac{n}{2}$,
$r(G)=\lfloor\frac{n}{2}\rfloor$. So $H(G)\geq r(G)$ and the equality holds
if and only if $n$ is even. In the following we assume $V(G-C) \ne \emptyset$. Then $T=G-u_i u_{i+1}$ is a spanning tree of $G$ for any edge $u_i u_{i+1}$ of $C$, and $r(T)\geq r(G)$. We study the following cases.

{\bf Case 1.}
Any longest path of $T$
contains all $l-1$ edges of $C-u_i u_{i+1}$, $\forall u_i u_{i+1} \in E(C)$. 

Note that in this case the degree of any vertex of $C$ is at least three. Otherwise there must exit an edge in $C$ such that one of its vertices has degree 2, and other has degree greater than 2. Without loss of generality, assume $u_1 u_2$ is such an edge, where $d_1=2,d_2\ge 3$. Then there must exist a vertex $w$ adjacent to $u_2$ other than $u_1$ and $u_3$. For $C-u_1 u_l$, any longest path of $T$ should contain $u_1 u_2 \cdots u_l$. However,  since the length of $w u_2 u_3 \cdots u_l$ is the same as the length of $u_1 u_2 u_3 \cdots u_l$, there must exist a longest path contains $w u_2 u_3 \cdots u_l$ which doesn't include edge $u_1u_2$. Thus a contradiction occurs.

Without loss of generality, in the following we assume the edge deleted from $C$ is $u_1 u_l$, i.e., $T=G-u_1 u_l$. Let $P=v_1v_2\cdots v_t$ denote a longest path of $T$, where $3 \le l < t$. So $r(T)=r(P)$ and $P$ contains all $l-1$ edges of $C-u_{1}u_{l}$. Note that it is impossible for $l=t$, i.e., $u_l \not= v_t$. For contradiction, suppose $u_l=v_t$. So $v_t$ is a vertex in $C$ with degree no less than 3. Then besides $u_1 v_t$ and $v_{t-1}v_t$, there should be another edge connecting $v_t$. It implies that $P$ is not the longest path; the longest path could be one more edge longer than $P$. For the same reason, it is impossible for $u_1=v_1$. In another words, the path $u_1 u_2 \cdots u_l$ is neither on the leftmost, nor on the rightmost of the longest path $P$. 

Now let's  add pendant edges $u_2u'_2, u_3u'_3, \cdots,
u_{l-1}u'_{l-1}$ to  $P$. Denote $T_1=P+u_2u'_2$, then
\begin{align*}
H(T_1)=H(P)+\frac{3}{10}.
\end{align*}

Let $T_2=T_1+u_3u'_3+\cdots+u_{l-1}u'_{l-1}$. By applying Lemma 1 iteratively, we obtain $H(T_2) > H(T_1)$. Now, let's denote $G'=T_2+u_1u_l$. By calculation, we get
\begin{align*}
H(G')=H(T_2)+
\left\{
\begin{array}{ll}
-\frac{2}{15},
& \mbox{if $u_1=v_2$ and $u_l=v_{t-1}$}; \\
-\frac{1}{15},
& \mbox{if $u_1=v_2$ and $u_l=v_{t-2}$}, \\
& \mbox{or $u_1=v_3$ and $u_l=v_{t-1}$};\\
0,
& \mbox{otherwise}.
\end{array}
\right.
\end{align*}

In all, we have 
\begin{align*}
H(G') \ge H(T_2)-\frac{2}{15}>H(T_1)-\frac{2}{15} =H(P)+\frac{3}{10} -\frac{2}{15}= H(P)+\frac{1}{6}.
\end{align*} 

Finally, we add all the residual edges $E(G\setminus G')$ to $G'$, step by step. Note that all of these edges are pendant edge. According to Lemma 1, we have $H(G)> H(G')$. Hence, $H(G) > H(G') > H(P)+\frac{1}{6}$. Since $P$ is a path, according to~\eqref{HarmonicPath}, $H(P)\ge r(P)-\frac{1}{6}$. Therefore, 
\begin{align*}
H(G)>H(P)+\frac{1}{6} \ge r(P)-\frac{1}{6}+\frac{1}{6}=r(P)=r(T)\ge r(G).
\end{align*}

{\bf Case 2.}
There is an edge $u_i u_{i+1}$ of $C$ such that $T=G-u_i u_{i+1}$ has a
longest path $P=v_1v_2\cdots v_t$ which contains at most $l-2$ edges
of $C-u_i u_{i+1}$. 

Obviously, $r(T)=r(P)$. There are three subcases to consider.

(i) $P$ and $C$ have no common vertex.

Let $P_0=w_1w_2\cdots w_p$ be the shortest path connecting $P$ and $C$,
where $w_1=v_i$ is a vertex on $P$ and $w_p$ is a vertex on $C$,
where $2\leq i\leq t-1$ since $P$ is a longest path of $T$.
Without loss of generality, we assume $w_p=u_1$. Let
$T_1=P+w_1w_2$, then $H(T_1)\geq H(P)+\frac{1}{6}$. Let
$T_2=T_1+w_2w_3+\cdots+w_{p-1}w_p$, then $H(T_2)> H(T_1)$ by
Lemma 1. Let $T_3=T_2+u_1u_2$, then $H(T_3)\geq H(T_2)+\frac{1}{2}$.
Let $T_4=T_3+u_1u_2+\cdots+u_{l-1}u_l$, then $H(T_4) > H(T_3)$ by
Lemma 1. Now, let $G'=T_4+u_lu_1$, then $H(G')\geq
H(T_4)+\frac{1}{30}$. So, $H(G')>H(P)+\frac{1}{6}$. Finally, let's add all the residual edges $E(G \setminus G')$ to $G'$ step by step. According to Lemma 1, we have $H(G) > H(G')$. Hence, 
\begin{align*}
H(G)>H(P)+\frac{1}{6} \geq r(P) = r(T) \geq r(G).
\end{align*}

(ii) $P$ and $C$ have exactly one common vertex.

Without loss of generality, we assume that $v_i=u_1$ is the unique
common vertex of $P$ and $C$, where $2\le i \le t-1$ since $P$ is a
longest path of $T$. Let $T_1=P+u_1u_2$, then $H(T_1)\geq
H(P)+\frac{1}{6}$. Let $T_2=T_1+u_2u_3$, then $H(T_2)\geq
H(T_1)+\frac{17}{30}\geq H(P)+\frac{22}{30}$. Let
$T_3=T_2+u_3u_4+\cdots+u_{l-1}u_l$, then $H(T_3)> H(T_2)$ by
Lemma 1. Now, let $G'=T_3+u_lu_1$, then $H(G')\geq
H(T_3)-\frac{1}{10}$. So, $H(G')> H(P)+\frac{19}{30}$. Finally,
let's add all the residual edges $E(G \setminus G')$ to $G'$ step by step. According to Lemma 1, we have $H(G) > H(G')$. Therefore, 
\begin{align*}
H(G)> H(P)+\frac{19}{30}>r(P)=r(T)\geq r(G).
\end{align*}

(iii) $P=v_1v_2\cdots v_t$ contains $s-1$ edges of $C$, where
$2\leq s\leq l-1$.

Without loss of generality, we assume that $u_1u_2\cdots
u_s=v_iv_{i+1}\cdots v_{i+s-1}$ for some $1\leq i<i+s-1\leq t$,
i.e., $P$ contains the edges $u_ju_{j+1}$ ($1\leq j\leq s-1$) of
$C$. Then $u_1\neq v_1$ or $u_s\neq v_t$ since $P$ is a longest
path of $T$. We assume that $u_1\neq v_1$. Let $T_1=P+u_su_{s+1}$,
then $H(T_1)\geq H(P)+\frac{7}{30}$. Let
$T_2=T_1+u_{s+1}u_{s+2}+\cdots+u_{l-1}u_l$, then $H(T_2)> H(T_1)$
by Lemma 1. Now, let $G'=T_2+u_1u_l$, then $H(G')\geq
H(T_2)-\frac{1}{30}$ in all cases. So, $H(G')>
H(P)+\frac{7}{30}-\frac{1}{30}=H(P)+\frac{1}{5}>r(P)$.
Finally, let's add all the residual edges $E(G\setminus G')$ to $G'$ step by step. According to Lemma 1, we have $H(G)> H(G')$. Hence, 
\begin{align*}
H(G)>r(P)=r(T)\geq r(G).
\end{align*}

\end{proof}

\section{The Harmonic index and the radius of a graph with cyclomatic number $k$}

We will need the following lemmas to prove our main result in this section.

\begin{lemma}
Let $f(x,y)=\frac{4}{x}-\frac{8}{x+1}+\frac{2}{x+2}+\frac{4}{y}
-\frac{8}{y+1}+\frac{2}{y+2}+\frac{2}{x+y}$,  $x,y \in \mathbb{N}^{+} \setminus \{1\}$, then $f(x,y) \ge -\frac{31}{105}$.
\end{lemma}

\begin{proof}
We first show that $f(x,y) \ge f(5,5)= -\frac{31}{105}$, when $x,y \in \mathbb{R}$ and $x\ge 5, y\ge 5$. Since
\begin{align*}
\frac{\partial f(x,y)}{\partial x} =& -\frac{4}{x^2} + \frac{8}{(x+1)^2} - \frac{2}{(x+2)^2} - \frac{2}{(x+y)^2}, \\
\frac{\partial^2 f}{\partial x \partial y} =&  \frac{4}{(x+y)^3} > 0,
\end{align*}
we know that $\frac{\partial f}{\partial x}$ is increasing in $y$. So, 
\begin{align*}
\frac{\partial f(x,y)}{\partial x} \ge \frac{\partial f(x,5)}{\partial x} =& -\frac{4}{x^2} + \frac{8}{(x+1)^2} - \frac{2}{(x+2)^2} - \frac{2}{(x+5)^2} \\
=& \frac{2(6 x^5 + 21x^4 - 96x^3 -527x^2 - 680x-200)}{x^2 (x+1)^2 (x+2)^2 (x+5)^2}.
\end{align*}

Denote $g(x)=6 x^5 + 21x^4 - 96x^3 -527x^2 - 680x-200$. It is easy to check that $g(x)>0$ for $x \ge 5$. Hence, $\frac{\partial f(x,y)}{\partial x} \ge \frac{\partial f(x,5)}{\partial x} >0$ for $x\ge 5$, which implies $f(x,y)$ is increasing in $x\ge 5$. 

Similarly, $f(x,y)$ is increasing in $y\ge 5$. Hence, $f(x,y)\ge f(5,5)=-\frac{31}{105}$.

Second, we compare the values of $f(x,y)$ at several discrete points, namely $f(2,2)=\frac{1}{6}, f(2,3)=-\frac{1}{30}, f(2,4)=-\frac{1}{10}, f(2,5)=-\frac{9}{70}, f(3,3)=-\frac{1}{5}, f(3,4)=-\frac{26}{105}, f(3,5)=-\frac{37}{140}, f(4,4)=-\frac{17}{60}, f(4,5)=-\frac{92}{315}, f(5,5)=-\frac{31}{105}$. Obviously, the minimum is $f(5,5)=-\frac{31}{105}$.

In all, $f(x,y) \ge -\frac{31}{105}$ when $x,y \in \mathbb{N}^{+} \setminus \{1\}$.
\end{proof}

\begin{lemma}
Let $G$ be a graph with cyclomatic number $k\geq 1$, and $v_i v_{i+1}$ is an edge in a cycle of $G$, then $H(G) \ge H(G - v_i v_{i+1})-\frac{31}{105}$.
\end{lemma}

\begin{proof} 
According to the definition of the Harmonic index, 
\begin{align*}
H(G)-H(G - v_i v_{i+1})=&\frac{2}{d_i+d_{i+1}}+\sum\limits_{v_j \in
N(v_i)\setminus\{v_{i+1}\}} \left( \frac{2}{d_i+d_j}-\frac{2}{d_i+d_j-1} \right) \\
&+ \sum\limits_{v_j\in
N(v_{i+1})\setminus\{v_i\}} \left( \frac{2}{d_{i+1}+d_j}-\frac{2}{d_{i+1}+d_j-1} \right).
\end{align*}

Since $v_i v_{i+1}$ is an edge in a cycle, there is a vertex $v_{i-1}$ adjacent to $v_i$ in the cycle with degree $d_{i-1}\geq 2$, and  a vertex $v_{i+2}$ adjacent to $v_{i+1}$ in the cycle with degree $d_{i+2} \geq 2$. So,
\begin{align*}
H(G) - H(G-v_i v_{i+1})\geq&
\frac{2}{d_i+d_{i+1}} + \left( \frac{2}{d_i+2}-\frac{2}{d_{i}+1} \right) + \left( d_i-2 \right) \left( \frac{2}{d_i+1}-\frac{2}{d_i} \right) \\
&+ \left( \frac{2}{d_{i+1}+2} - \frac{2}{d_{i+1}+1} \right) + \left(d_{i+1}-2 \right) \left( \frac{2}{d_{i+1}+1}-\frac{2}{d_{i+1}} \right)\\
=&\frac{4}{d_i}-\frac{8}{d_i+1}+\frac{2}{d_i+2}+\frac{4}{d_{i+1}}-\frac{8}{d_{i+1}+1}+\frac{2}{d_{i+1}+2}+\frac{2}{d_i+d_{i+1}} \\
\ge & -\frac{31}{105}  \,\,\,\,\,\,\,\,(\mbox{Lemma 2}).
\end{align*}
\end{proof}

\begin{theorem}
Let $G$ be a graph with cyclomatic number $k\geq 1$. Then
$H(G) \ge r(G) - \frac{31}{105}(k-1)$. In particular, $H(G)>r(G)-1$ for a graph with cyclomatic number no more than 4.
\end{theorem}

\begin{proof} 
We first prove the case $k\ge 2$. Since the cyclomatic number of graph $G$ is $k\ge 2$, there exists a sequense of edges $e_1,\dots,e_{k}$ such that the cyclomatic number of graph $G_i=G-\{e_1,e_2,\cdots,e_{i}\}$ is $k-i$, where $1\le i \le k$. In particular, $G_{k-1}$ is a spanning unicylic subgraph of $G$. Note that $r(G)\leq r(G_1)\leq r(G_2)\leq\cdots\leq r(G_{k-1})$. According to Lemma 3 and Theorem 2, we have
\begin{align*}
H(G) \ge & H(G_1)-\frac{31}{105} \ge H(G_2)-\frac{31}{105}-\frac{31}{105} \ge \cdots \ge H(G_{k-1})-\frac{31}{105}(k-1)\\
\geq&r(G_{k-1})-\frac{31}{105}(k-1)\geq r(G)-\frac{31}{105}(k-1).
\end{align*}

Note that this is in accord with Theorem 2. That is, when $G$ is unicyclic graph, i.e., $k=1$, $H(G) \ge r(G)$. Therefore, the theorem is true for $k\ge 1$.

In particular, $H(G) \ge r(G)-1$ for $1\le k \le 4$.
\end{proof}

\section{Conclusion}

We focus on the conjectures of Randi\'{c} index and graph radius. These conjectures have been opened for a long time. We improve and strengthen the known results on the conjectures by studing the relationship between the Harmonic index and graph radius. It is interesting to know whether or not the conjectures are true for more general graphs. In particular, could the techniques in this paper be extended to studying more general graphs? It is intriguing to know whether the following conjecture is true.

\begin{conjecture}
For all connected graphs $G$ except even paths, $H(G)\geq r(G)$.
\end{conjecture}

\end{document}